\documentclass[12pt]{amsart}
\usepackage{amssymb,amsmath,amsthm,amssymb}
\usepackage{mathtools}
\usepackage{graphicx}
\usepackage{comment}
\usepackage{hyperref}
\usepackage{enumerate}
\usepackage{color}
\usepackage{epstopdf}
\renewcommand{\epsilon}{\varepsilon}

\DeclareMathOperator{\area}{area}

\DeclareMathOperator{\sech}{sech}

\def\R{\mathbb{R}}

\def\N{\mathbb{N}}

\def\a{\alpha}
\def\e{\epsilon}
\def\t{\tau}

\def\th{\theta}

\def\o{\omega}

\def\k{\kappa}

\def\tilde{\widetilde}

\def\ov{\overline}

\newcommand{\pd}{\partial}
\newcommand{\cd}{\nabla}

\def\ba #1\ea {\begin{align} #1\end{align}}
\def\bann #1\eann {\begin{align*} #1\end{align*}}
\def\ben #1\een {\begin{enumerate} #1\end{enumerate}}
\def\bi #1\ei {\begin{itemize}\renewcommand\labelitemi{--} #1\end{itemize}}

\newcommand{\inner}[2]{\left\langle#1,#2\right\rangle} 

\newtheorem{theorem}{Theorem}[section]
\newtheorem*{theorem*}{Theorem}
\newtheorem{lemma}[theorem]{Lemma}
\newtheorem{proposition}[theorem]{Proposition}
\newtheorem{remark}[theorem]{Remark}
\newtheorem{corollary}[theorem]{Corollary}

\newtheorem*{conjecture*}{Conjecture}

\newtheorem{claim}{Claim}[theorem]
\newtheorem*{claim*}{Claim}

\newtheoremstyle{TheoremNum}
        {\topsep}{\topsep}              
        {\itshape}                      
        {}                              
        {\bfseries}                     
        {.}                             
        { }                             
        {\thmname{#1}\thmnote{ \bfseries #3}}
    \theoremstyle{TheoremNum}

\author{Theodora Bourni}
\author{Mat Langford}
\address{Department of Mathematics, University of Tennessee Knoxville, Knoxville TN, 37996-1320}
\email{tbourni@utk.edu, mlangford@utk.edu}
\address{School of Mathematical and Physical Sciences, The University of Newcastle, Newcastle, NSW, Australia, 2308}
\email{mathew.langford@newcastle.edu.au}

\date{\today}

\title[Convex ancient free boundary csf in the disc.]{Classification of convex ancient free boundary curve shortening flows in the disc.}

\begin{document}

\begin{abstract}
We classify convex ancient curve shortening flows in the disc with free boundary on the circle. 
\end{abstract}

\maketitle

\setcounter{tocdepth}{1}
\tableofcontents

\section{Introduction}

Curve shortening flow is the gradient flow of length for regular curves. It models the evolution of grain boundaries \cite{Mullins,vonNeumann} and the shapes of worn stones \cite{Firey} in two dimensions, and has been exploited in a multitude of further applications (see, for example, \cite{MR1813971}).

The evolution of closed planar curves by curve shortening was initiated by Mullins \cite{Mullins} and was later taken up by Gage \cite{Ga84} and Gage--Hamilton \cite{GaHa86}, who proved that closed convex curves remain convex and shrink to ``round'' points in finite time. Soon after, Grayson showed that closed embedded planar curves become convex in finite time under the flow, thereafter shrinking to round points according to the Gage--Hamilton theorem. Different proofs of these results were discovered later by others \cite{An12,AndrewsBryanisoperimetric,AndrewsBryandistance,MR1369140,Huisken96}. 

Ancient solutions to geometric flows (that is, solutions defined on backwards-infinite time-intervals) are important from an analytical standpoint as they model singularity formation \cite{MR1375255}. They also arise in quantum field theory, where they model the ultraviolet regime in certain Dirichlet sigma models \cite{Bakas2}. They have generated a great deal of interest from a purely geometric standpoint due to their symmetry and rigidity properties. Indeed, ancient solutions to curve shortening flow of convex planar curves have been classified through the work of Daskalopoulos--Hamilton--\v Se\v sum \cite{DHS} and the authors in collaboration with Tinaglia \cite{BLT3}. Bryan and Louie \cite{BrLo} proved that the shrinking parallel is the only convex ancient solution to curve shortening flow on the two-sphere, and Choi and Mantoulidis showed that it is the only embedded ancient solution on the two-sphere with uniformly bounded length 
\cite{ChMa}.

The natural Neumann boundary value problem for curve shortening flow, called the \emph{free boundary problem}, asks for a family of curves whose endpoints lie on (but are free to move on) a fixed barrier curve which is met by the solution curve orthogonally. Study of the free boundary problem was initiated by Huisken \cite{HuiskenNonparametric} 
and further developed by Stahl \cite{Stahl2,Stahl1}. In particular, Stahl proved that convex curves with free boundary on a smooth, convex, locally uniformly convex barrier remain convex and shrink to a point on the barrier curve. 

The analysis of ancient solutions to free boundary curve shortening flow remains in its infancy. Indeed, to our knowledge, the only examples previously known seem to be those inherited from closed or complete examples (one may restrict the shrinking circle, for example, to the upper halfplane).

We provide here a classification of convex\footnote{A free boundary curve in the open disc $B^2$ is \emph{convex} if it bounds a convex region in $B^2$ and \emph{locally uniformly convex} if it is of class $C^2$ and its curvature is positive.} ancient free boundary curve shortening flows in the disc. 

\begin{theorem}\label{thm:classification}
Modulo rotation about the origin and translation in time, there exists exactly one convex, locally uniformly convex ancient solution to free boundary curve shortening flow in the disc. It converges to the point $(0,1)$ as $t\to 0$ and smoothly to the segment $[-1,1]\times\{0\}$ as $t\to-\infty$. It is invariant under reflection across the $y$-axis. As a graph over the $x$-axis, it satisfies
\[
\mathrm{e}^{\lambda^2t}y(x,t)\to A\cosh(\lambda x)\;\;\text{uniformly in $x$ as}\;\; t\to-\infty
\]
for some $A>0$, where $\lambda$ is the solution to $\lambda\tanh\lambda=1$.
\end{theorem}

Theorem \ref{thm:classification} is a consequence of Propositions \ref{prop:existence}, \ref{prop:linear analysis}, and \ref{prop:uniqueness} proved below. Note that it is actually a classification of all convex ancient solutions, since the strong maximum principle and the Hopf boundary point lemma imply that any convex solution to the flow is either a stationary segment (and hence a bisector of the disc by the free boundary condition) or is locally uniformly convex at interior times.

A higher dimensional counterpart of Theorem \ref{thm:classification} will be treated in a forthcoming paper.

Another natural setting in which to seek ancient solutions is within the class of soliton solutions. Since free boundary curve shortening flow in the disc is invariant under ambient rotations, one might expect to find rotating solutions. In \S \ref{sec:no rotators}, we provide a short proof 
that none exist.

\begin{theorem}\label{thm:no rotators}
There exist no proper rotating solutions to free boundary curve shortening flow in the disc.
\end{theorem}

\subsection*{Acknowledgements}

We wish to thank Jonathan Zhu for sharing his thoughts on the problem.

TB was supported through grant 707699 of the Simons Foundation and grant DMS-2105026 of the National Science Foundation. ML was supported through an Australian Research Council DECRA fellowship (grant DE200101834).

\section{Existence}\label{sec:existence}

Our first goal is the explicit construction of a non-trivial ancient free boundary curve shortening flow in the disc. It will be clear from the construction that the solution is reflection symmetric about the vertical axis, emerges at time negative infinity from the horizontal bisector, and converges at time zero to the point $(0,1)$. 
We shall also prove an estimate for the height of the constructed solution (which will be needed to prove its uniqueness).

\subsection{Barriers}


Given $\theta\in (0,\frac{\pi}{2})$, denote by $\mathrm{C}_\theta$ the circle centred on the $y$-axis which meets $\pd B^2$ orthogonally at $(\cos\theta,\sin\theta)$. That is,
\begin{equation}\label{eq:critical circle}
\mathrm{C}_\theta\doteqdot \left\{(x,y)\in \R^2:x^2+\left(\csc\theta-y\right)^2=\cot^2\theta\right\}\,.
\end{equation}
If we set
\[
\theta^-(t)\doteqdot \arcsin \mathrm{e}^t\;\;\text{and}\;\; \theta^+(t)\doteqdot \arcsin \mathrm{e}^{2t}\,,
\]
then we find that the inward normal speed of $\mathrm{C}_{\theta^-(t)}$ is no greater than its curvature $\kappa^-$ and the inward normal speed of $\mathrm{C}_{\theta^+(t)}$ is no less than its curvature $\kappa^+$. The maximum principle and the Hopf boundary point lemma then imply that 
\begin{proposition}\label{prop:circle barriers}
a solution to free boundary curve shortening flow in $B^2$ which lies below (resp. above) the circle $\mathrm{C}_{\theta_0}$ at time $t_0$ lies below $\mathrm{C}_{\theta^+(t^+_0+t-t_0)}$ (resp. above $\mathrm{C}_{\theta^-(t^-_0+t-t_0)}$) for all $t>t_0$, where $2t^+_0=\log\sin\theta_0$ (resp. $t^-_0=\log\sin\theta_0$).
\end{proposition}

Consider now the shifted scaled Angenent oval $\{\mathrm{A}^\lambda_t\}_{t\in(-\infty,0)}$, where
\[
\mathrm{A}^\lambda_t\doteqdot \left\{(x,y)\in\R\times(0,\tfrac{\pi}{2\lambda}):
\sin(\lambda y)=\mathrm{e}^{\lambda^2t}\cosh\big(\lambda x\big)\right\}\,.
\]
This evolves by curve shortening flow and satisfies
\[
\nu_\lambda(\cos\theta,\sin\theta)\cdot (\cos\theta,\sin\theta)=\frac{\cos\theta\tanh(\lambda\cos\theta)-\sin\theta\cot(\lambda \sin\theta))}{\sqrt{\tanh^2(\lambda\cos\theta)+\cot^2(\lambda\sin\theta))}}
\]
at a point $(\cos\theta,\sin\theta)\in \pd B^2$.

\begin{lemma}\label{lem:Angenent BC}
For each $\theta\in(0,\frac{\pi}{2})$, there is a unique $\lambda(\theta)\in (0,\frac{\pi}{2\sin\theta})$ such that
\[
\nu_{\lambda(\theta)}(\cos\theta,\sin\theta)\cdot (\cos\theta,\sin\theta)=0\,.
\]
Given $\theta,\theta_0\in(0,\frac{\pi}{2})$ with $\theta>\theta_0$,
\[
\nu_{\lambda(\theta_0)}(\cos\theta,\sin\theta)\cdot (\cos\theta,\sin\theta)<0\,.
\]
\end{lemma}
\begin{proof}
Define
\[
f(\lambda,\theta)\doteqdot \cos\theta\tanh(\lambda\cos\theta)-\sin\theta\cot(\lambda \sin\theta))\,.
\]
Observe that
\[
\lim_{\lambda\searrow 0}f(\lambda,\theta)=-\infty\,,\;\; \lim_{\lambda\nearrow\frac{\pi}{2\sin\theta}}f(\lambda,\theta)=\cos\theta\tanh(\tfrac{\pi}{2}\cot\theta)>0\,,
\]
and 
\begin{equation}\label{eq:f lambda monotonicity}
\frac{\pd f}{\pd\lambda}=\cos^2\theta(1-\tanh^2(\lambda\cos\theta))+\sin^2\theta(1+\cot^2(\lambda\sin\theta))>0\,.
\end{equation}
The first claim follows.

Next observe that
\bann
\frac{\pd f}{\pd\theta}={}&-\sin\theta\tanh(\lambda\cos\theta)-\lambda\cos\theta\sin\theta\sech^2(\lambda\cos\theta)\\
{}&-\cos\theta\cot(\lambda\sin\theta)+\lambda\sin\theta\cos\theta\csc^2(\lambda\sin\theta)\,.
\eann
Given $\theta\in(0,\frac{\pi}{2})$, we obtain, at the unique zero $\lambda\in(0,\frac{\pi}{2\sin\theta})$ of $f(\cdot,\theta)$,
\bann
\frac{\pd f}{\pd\theta}={}&-\sin\theta\tan\theta\cot(\lambda\sin\theta)-\lambda\cos\theta\sin\theta(1-\tan^2\theta\cot^2(\lambda\sin\theta))\\
{}&-\cos\theta\cot(\lambda\sin\theta)+\lambda\sin\theta\cos\theta\csc^2(\lambda\sin\theta)\\
={}&-\sec\theta\cot(\lambda\sin\theta)\left(1-\lambda\sin\theta\cot(\lambda\sin\theta)\right).
\eann
Since $Y\cot Y<1$ for $Y\in(0,\frac{\pi}{2})$, this is less than zero. The second claim follows.
\end{proof}

The maximum principle and the Hopf boundary point lemma now imply the following.

\begin{proposition}\label{prop:Angenent oval barriers}
Let $\{\Gamma_t\}_{t\in [\alpha,\omega)}$ be a solution to free boundary curve shortening flow in $B^2$. Suppose that $\lambda\le\lambda(\theta_\alpha)$, where $\theta_\alpha$ denotes the smaller, in absolute value, of the two turning angles to $\Gamma_\alpha$ at its boundary. 
If $\Gamma_\alpha$ lies above $\mathrm{A}^{\lambda}_{s}$, then $\Gamma_t$ lies above $\mathrm{A}^{\lambda}_{s+t-\alpha}$ for all $t\in(\alpha,\omega)\cap(-\infty,\alpha-s)$.
\end{proposition}
\begin{proof}
By the strong maximum principle, the two families of curves can never develop contact at an interior point. Since the families are monotonic, they cannot develop boundary contact at a boundary point $(\cos\theta,\sin\theta)$ with $\vert\theta\vert\le \theta_\alpha$. On the other hand, since $\lambda\le\lambda(\theta_\alpha)$, \eqref{eq:f lambda monotonicity} implies that
\[
f(\lambda,\theta_\alpha)\le f(\lambda_\alpha,\theta_\alpha)=0\,,
\]
and hence, by the argument of Lemma \ref{lem:Angenent BC},
\[
f(\lambda,\theta)\le 0\;\; \text{for}\;\; \theta\ge \theta_\alpha\,.
\]
So the Hopf boundary point lemma implies that no boundary contact can develop for $\theta\ge \theta_\alpha$ either. 
\end{proof}

\begin{remark}
Since $s\cot s\to 1$ as $s\to 0$, $f(\lambda,\theta)$ is non-negative at $\theta=0$ so long as $\lambda\ge \lambda_0$, where $\lambda_0\tanh\lambda_0=1$. 
\end{remark}

\subsection{Old-but-not-ancient solutions}
For each $\rho>0$, choose a curve $\Gamma^\rho$ in $\overline B{}^2$ with the following properties:
\begin{itemize}
\item $\Gamma^\rho$ meets $\pd B^2$ orthogonally at $(\cos\rho,\sin\rho)$,
\item $\Gamma^\rho$ is reflection symmetric about the $y$-axis,
\item $\Gamma^\rho\cap B^2$ is the relative boundary of a convex region $\Omega^\rho\subset B^2$, and
\item $\kappa^\rho_s>0$ in $B^2\cap\{x>0\}$.
\end{itemize}

For example, we could take $\Gamma^\rho\doteqdot \mathrm{A}^{\lambda_\rho}_{t_\rho}\cap B^2$, where $\lambda_\rho>\lambda_0$ and $t_\rho$ are (uniquely) chosen so that
\[
\cos\rho\tanh(\lambda_\rho\cos\rho)-\sin\rho\cot(\lambda_\rho\sin\rho))=0
\]
and
\[
-t_\rho=\lambda_\rho^{-2}\log\left(\frac{\cosh(\lambda_\rho\cos\rho)}{\sin(\lambda_\rho\sin\rho)}\right)\,.
\]
Observe that the circle $\mathrm{C}_{\theta_\rho}$ defined by
\[
\sin\theta_\rho=\frac{2\sin\rho}{1+\sin^2\rho}
\]
is tangent to the line $y=\sin\rho$, and hence lies above $\Gamma^\rho$.

Work of Stahl \cite{Stahl1,Stahl2} now yields the following \emph{old-but-not-ancient solutions}.

\begin{lemma}\label{eq:old-but-not-ancient}
For each $\rho\in(0,\frac{\pi}{2})$, there exists a smooth solution\footnote{Given by a one parameter family of immersions $X:[-1,1]\times [\alpha_\rho,0)\to \overline B{}^2$ satisfying $X\in C^\infty([-1,1]\times (\alpha_\rho,0))\cap C^{2+\beta,1+\frac{\beta}{2}}([-1,1]\times [\alpha_\rho,0))$ for some $\beta\in(0,1)$.} $\{\Gamma^\rho_t\}_{t\in[\alpha_\rho,0)}$ to curve shortening flow with $\Gamma^\rho_{\alpha_\rho}=\Gamma^\rho$ which satisfies the following properties:
\begin{itemize}
\item $\Gamma^\rho_t$ meets $\pd B^2$ orthogonally for each $t\in(\alpha_\rho,0)$,
\item $\Gamma^\rho_t$ is convex and locally uniformly convex for each $t\in(\alpha_\rho,0)$,
\item $\Gamma^\rho_t$ is reflection symmetric about the $y$-axis for each $t\in(\alpha_\rho,0)$,
\item $\Gamma_t^\rho\to (0,1)$ uniformly as $t\to 0$,
\item $\kappa^\rho_s>0$ in $B^2\cap\{x>0\}$, and
\item $\alpha_\rho<\frac{1}{2}\log\left(\frac{2\sin\rho}{1+\sin^2\rho}\right)\to-\infty$ as $\rho\to 0$.
\end{itemize}
\end{lemma}
\begin{proof}
Existence of a maximal solution to curve shortening flow out of $\Gamma^\rho$ which meets $\pd B^2$ orthogonally was proved by Stahl \cite[Theorem 2.1]{Stahl1}. Stahl also proved that this solution remains convex and locally uniformly convex and shrinks to a point on the boundary of $B^2$ at the final time (which is finite) \cite[Proposition 1.4]{Stahl2}. We obtain $\{\Gamma_t^\rho\}_{t\in[\alpha_\rho,0)}$ by time-translating Stahl's solution. 


By uniqueness of solutions 
$\Gamma^\rho_t$ remains reflection symmetric about the $y$-axis for $t\in(\alpha_\rho,0)$, so the final point is $(0,1)$.

The reflection symmetry also implies that $\kappa^\rho_s=0$ at the point $p_t\doteqdot \Gamma^\rho_t\cap\{x=0\}$ for all $t\in[\alpha_\rho,0)$. By \cite[Proposition 2.1]{Stahl2}, $\kappa^\rho_s=\kappa^\rho>0$ at the boundary point $q_t\doteqdot \pd\Gamma^\rho_t\cap\{x>0\}$ for all $t\in(\alpha_\rho,0)$. Applying Sturm's theorem \cite{MR953678} to $\kappa^\rho_s$, we thus find that $\kappa^\rho_s>0$ on $\Gamma^\rho_t\cap B^2\cap\{x>0\}$ for all $t\in(\alpha_\rho,0)$.

Since $\mathrm{C}_{\theta_\rho}\subset\Omega^\rho$, the final property follows from Proposition \ref{prop:circle barriers}.
\end{proof}

We now fix $\rho>0$ and drop the super/subscript $\rho$. Set
\[
\underline\kappa(t)\doteqdot \min_{\Gamma_t}\kappa=\kappa(p_t)\;\;\text{and}\;\;\overline\kappa(t)\doteqdot \max_{\Gamma_t}\kappa=\kappa(q_t)\,,
\]
and define $\underline y(t)$, $\overline y(t)$ and $\overline\theta(t)$ by
\[
p_t=(0,\underline y(t))\,,\;\;q_t=(\cos\overline\theta(t),\sin\overline\theta(t))\,,\;\;\text{and}\;\;\overline y(t)=\sin\overline\theta(t)\,.
\]

\begin{lemma}\label{lem:crude estimates}
Each old-but-not-ancient solution satisfies
\begin{equation}\label{eq:kappa bounds}
\underline\kappa\le\tan\overline\theta\le \overline\kappa\,,
\end{equation}
\begin{equation}\label{sin theta upper bound}
\sin\overline\theta\le e^t\,,
\end{equation}
and
\begin{equation}\label{eq:min y lower bound}
\frac{\sin\overline\theta}{1+\cos\overline\theta}\le \underline y\le \sin\overline\theta\,.
\end{equation}
\end{lemma}
\begin{proof}
To prove the lower bound for $\overline \k$, it suffices to show that the circle $\mathrm{C}_{\overline \theta(t)}$ (see \eqref{eq:critical circle}) lies locally below $\Gamma_t$ near $q_t$. 
If this is not the case, then, locally around $q_t$, $\Gamma_t$ lies below $\mathrm{C}_{\overline \theta(t)}$ and hence $\kappa(q_t)\le \tan\overline\theta(t)$. But then we can translate $\mathrm{C}_{\overline \theta(t)}$ downwards until it touches $\Gamma_t$ from below in an interior point at which the curvature must satisfy $\k\ge \tan\overline\theta(t)$. This contradicts the unique maximization of the curvature at $q_t$.

The estimate \eqref{sin theta upper bound} now follows by integrating the inequality
\[
\frac{d}{dt}\sin\overline\theta=\cos\overline\theta\,\overline\kappa\ge\sin\overline\theta
\]
between any initial time $t$ and the final time $0$ (at which $\overline\theta=\frac{\pi}{2}$ since the solution contracts to the point $(0,1)$).

The upper bound for $\underline y$ follows from convexity and the boundary condition $\overline y=\sin\overline\theta$. To prove the lower bound, we will show that the circle $\mathrm{C}_{\overline \theta(t)}$ lies nowhere above $\Gamma_t$. Suppose that this is not the case. Then, since $\mathrm{C}_{\overline \theta(t)}$ lies locally below $\Gamma_t$ near $q_t$, we can move $\mathrm{C}_{\overline\theta(t)}$ downwards until it is tangent from below to a point $p'_t$ on $\Gamma_t\cap\{x\ge 0\}$, at which we must have $\k\ge \tan\overline\theta(t)$. But then, since $\kappa_s\ge 0$ in $\{x>0\}$, we find that $\k\ge \tan \overline\theta(t)$ for all points between $p'_t$ and $q_t$. But this implies that this whole arc (including $p'_t$) lies above $\mathrm{C}_{\overline\theta(t)}$, a contradiction. 

To prove the upper bound for $\underline\kappa$, fix $t$ and consider the circle $C$ centred on the $y$-axis through the points $p_t$ and $q_t$. Its radius is $r(t)$, where
\[
r\doteqdot \frac{\cos^2\overline\theta+(\sin\overline\theta-\underline y)^2}{2(\sin\overline\theta-\underline y)}\,.
\]
We claim that $\Gamma_t$ lies locally below $C$ near $p_t$. Suppose that this is not the case. Then, by the symmetry of $\Gamma_t$ and $C$ across the $y$-axis, $\Gamma_t$ lies locally above $C$ near $p_t$. This implies two things: first, that
\[
\kappa(p_t)\ge r^{-1},
\]
and second, that, by moving $C$ vertically upwards, we can find a point $p'_t$ (the final point of contact) which satisfies
\[
\kappa(p'_t)\le r^{-1}\,.
\] 
These two inequalities contradict the (unique) minimization of $\kappa$ at $p_t$. We conclude that
\[
\underline\kappa\le\frac{2(\sin\overline\theta-\underline y)}{\cos^2\overline\theta+(\sin\overline\theta-\underline y)^2}\le \tan\overline\theta
\]
due to the lower bund for $\underline y$.
\end{proof}

\begin{remark}\label{rmk:Gage's lemma}
If we parametrize by turning angle $\theta\in[-\overline\theta,\overline\theta]$, so that
\[
\tau=(\cos\theta,\sin\theta)\,,
\]
then the estimates \eqref{eq:kappa bounds} are also easily obtained from the monotonicity of $\kappa$ and the formulae
\begin{equation}\label{eq:curvedef}
x(\theta)=x_0+\int_0^{\theta} \frac{\cos u}{\kappa(u)}du\;\;\text{and}\;\;y(\theta)=y_0+\int_{0}^{\theta}\frac{\sin u}{\kappa(u)}du\,.
\end{equation}
\end{remark}

\subsection{Taking the limit}

\begin{proposition}\label{prop:existence}
There exists a convex, locally uniformly convex ancient curve shortening flow in the disc with free boundary on the circle.
\end{proposition}
\begin{proof}
For each $\rho>0$, consider the old-but-not-ancient solution $\{\Gamma^\rho_t\}_{t\in[\alpha_\rho,0)}$, $\Gamma^\rho_t=\pd \Omega^\rho_t$, constructed in Lemma \ref{eq:old-but-not-ancient}. By \eqref{sin theta upper bound}, $\Omega^\rho_t$ contains $\mathrm{C}_{\omega(t)}\cap B^2$, where $\omega(t)\in(0,\frac{\pi}{2})$ is uniquely defined by
\[
\frac{1-\cos\omega(t)}{\sin\omega(t)}=\mathrm{e}^t\,.
\]
If we represent $\Gamma^\rho_t$ as a graph $x\mapsto y^\rho(x,t)$ over the $x$-axis, then convexity and the boundary condition imply that $\vert y^\rho_x\vert\le \tan\omega$. Since $\omega(t)$ is independent of $\rho$, Stahl's (global in space, interior in time) Ecker--Huisken type estimates \cite{Stahl1} imply uniform-in-$\rho$ bounds for the curvature and its derivatives. So the limit
\[
\{\Gamma^\rho_t\}_{t\in[\alpha_\rho,0)}\to \{\Gamma_t\}_{t\in(-\infty,0)}
\]
exists in $C^\infty$ (globally in space on compact subsets of time) and the limit $\{\Gamma_t\}_{t\in(-\infty,0)}$ satisfies curve shortening flow with free boundary in $B^2$. On the other hand, since  $\{\Gamma^\rho_t\}_{t\in(\alpha_\rho,0)}$ contracts to $(0,1)$ as $t\to 0$, (the contrapositive of) Proposition \ref{prop:circle barriers} implies that $\Gamma^\rho_t$ must intersect the closed region enclosed by $\mathrm{C}_{\theta^+(t)}$ for all $t<0$. It follows that $\Gamma_t$ must intersect the closed region enclosed by $\mathrm{C}_{\theta^+(t)}$ for all $t<0$. Since each $\Gamma_t$ is the limit of convex boundaries
, each is convex
. It follows that $\Gamma_t$ converges to $(0,1)$ as $t\to 0$ and, by \cite[Corollary 4.5]{Stahl1}, that $\Gamma_t$ is locally uniformly convex for each $t$.
\end{proof}

\subsection{Asymptotics for the height}

For the purposes of this section, we fix an ancient solution $\{\Gamma_t\}_{(-\infty,0)}$ obtained as in Proposition \ref{prop:existence} by taking a sublimit as $\lambda\searrow \lambda_0$ of the specific old-but-not ancient solutions $\{\Gamma^\lambda_t\}_{t\in[\alpha_\lambda,0)}$ corresponding to $\Gamma^\lambda_{\alpha_\lambda}=\mathrm{A}^{\lambda}_{t_\lambda}\cap B^2$, $t_\lambda$ being the time at which $\{\mathrm{A}^\lambda_t\}_{t\in(-\infty,0)}$ meets $\pd B^2$ orthogonally. The asymptotics we obtain for this solution will be used to prove its uniqueness. 

We will need to prove that the limit $\lim_{t\to-\infty}\mathrm{e}^{-\lambda_0^2t}\underline y(t)$ exists in $(0,\infty)$. The following speed bound will imply that it exists in $[0,\infty)$. 

\begin{lemma} \label{lem:lower curvature estimate}
The ancient solution $\{\Gamma_t\}_{(-\infty,0)}$ satisfies
\begin{equation}\label{eq:curvature lower}
\frac{\kappa}{\cos\theta}\ge \lambda_0\tan(\lambda_0y)\,.
\end{equation}
\end{lemma}
\begin{proof}
It suffices to prove that $\frac{\kappa}{\cos\theta}\ge \lambda \tan(\lambda y)$ on each of the old-but-not-ancient solutions $\{\Gamma^\lambda_t\}_{t\in[\alpha_\lambda,0)}$. Note that equality holds 
on the initial timeslice $\Gamma^\lambda_{\alpha_\lambda}=\mathrm{A}^\lambda_{t_\lambda}$.

Given any $\mu<\lambda$, set $u\doteqdot\mu\tan(\mu y)$ and $v\doteqdot x_s=\cos\theta=\inner{\nu}{e_2}$. Observe that
\[
u_s=\mu^2\sec^2(\mu y)\sin\theta\,,\;\; (\pd_t-\Delta)u=-2\mu^2\sec^2(\mu y)\sin^2\theta u\,,
\]
\[
v_s=-\kappa\sin\theta\;\;\text{and}\;\;(\pd_t-\Delta)v=\kappa^2v\,.
\]

At an interior maximum of $\frac{uv}{\kappa}$ we observe that
\[
\frac{\cd\kappa}{\kappa}=\frac{\cd u}{u}+\frac{\cd v}{v}
\]
and hence
\ba
0\le(\pd_t-\Delta)\frac{uv}{\kappa}={}&\frac{uv}{\kappa}\left(\frac{(\pd_t-\Delta)u}{u}-2\inner{\frac{\cd u}{u}}{\frac{\cd v}{v}}\right)\nonumber\\
={}&2\mu^2\sec^2(\mu y)\sin^2\theta\left(1-\frac{uv}{\kappa}\right).\label{eq:interior case}
\ea

At a (without loss of generality right) boundary maximum of $\frac{uv}{\kappa}$, we have $y_s=y$ and $\kappa_s=\kappa$, and hence
\ba
\left(\frac{uv}{\kappa}\right)_s={}&\frac{uv}{\kappa}\left(\frac{u_s}{u}+\frac{v_s}{v}-\frac{\kappa_s}{\kappa}\right)\nonumber\\
={}&\frac{uv}{\kappa}\left(\frac{\sec^2(\mu y)\mu y}{\tan{\mu y}}-\kappa\frac{y}{v}-1\right)\nonumber\\
\le{}&\left(\frac{uv}{\kappa}-1\right)\tan(\mu y)\mu y\,.\label{eq:boundary case}
\ea

We may now conclude that $\max_{\overline\Gamma^\lambda_t}\frac{uv}{\kappa}$ remains less than one. Indeed, if $\frac{uv}{\kappa}$ ever reaches $1$, then there must be a first time $t_0>0$ and a point $x_0\in\overline\Gamma_t$ at which this occurs (note that $uv/\kappa$ is continuous on $\overline\Gamma_t$ up to the initial time). The point $x_0$ cannot be an interior point, due to \eqref{eq:interior case}, and it cannot be a boundary point, due to \eqref{eq:boundary case} and the Hopf boundary point lemma. We conclude that
\[
\frac{\kappa}{\cos\theta}\ge \mu\tan(\mu y)
\]
on $\{\Gamma^\lambda_t\}_{t\in[\alpha_\lambda,0)}$ for all $\mu<\lambda$. Now take $\mu\to\lambda$.
\end{proof}

If we parametrize $\Gamma_t$ as a graph $x\mapsto y(x,t)$ over the $x$-axis, then \eqref{eq:curvature lower} yields
\bann
\big(\sin(\lambda_0y)\big)_t={}&\lambda_0\cos(\lambda_0y)\kappa\sqrt{1+\vert y_x\vert^2}=\lambda_0\cos(\lambda_0y)\frac{\kappa}{\cos\theta}\ge\lambda_0\sin(\lambda_0y)
\eann
and hence
\begin{equation}\label{eq:rescaled height monotonicity}
\left(\mathrm{e}^{-\lambda_0^2t}\sin(\lambda_0y(x,t))\right)_t\ge 0\,.
\end{equation}

In particular, the limit
\[
A(x)\doteqdot\lim_{t\to-\infty}\mathrm{e}^{-\lambda_0^2t}y(x,t)
\]
exists in $[0,\infty)$ for each $x\in(-1,1)$, as claimed. 

We next prove that the limit is positive. 
The following lemma will be used to prove the requisite speed bound.

\begin{lemma}\label{lem:gradient estimate}
There exist $T>-\infty$ and $C<\infty$ such that
\begin{equation}\label{eq:curvature estimate}
\overline\kappa\le Ce^{t}\;\;\text{for all}\;\; t<T\,.
\end{equation}
\end{lemma}
\begin{proof}
We will prove the estimate for each old-but-not-ancient solution $\{\Gamma^\lambda_t\}_{t\in(\alpha_\lambda,0)}$. We first prove a crude gradient estimate of the form
\begin{equation}\label{eq:crude kappa gradient estimate}
\vert\kappa_s\vert\le 2\kappa
\end{equation}
for $t$ sufficiently negative. It will suffice to prove that
\begin{equation}\label{eq:1k}
\vert\kappa_s\vert-\kappa+\inner{\gamma}{\nu}\le 0\,.
\end{equation}
Indeed, since $\inner{\gamma}{\nu}_s=\kappa\inner{\gamma}{\tau}$ has the same sign as the $x$-coordinate, we may estimate, as in \eqref{eq:curvature lower},
\begin{equation}\label{eq:gamma dot nu le kappa}
\vert\!\inner{\gamma}{\nu}\!\vert\le\vert\!\inner{\gamma}{\nu}\!\vert_{x=0}\le\lambda^{-2}\kappa|_{x=0}=\lambda^{-2}\min_{\Gamma_t}\kappa\le \kappa\,.
\end{equation}

For $\lambda$ sufficiently close to $\lambda_0$, we have $\kappa|_{t=\alpha_\lambda}<1/2$. Denote by $T^\lambda$ the first time at which $\kappa$ reaches $1/2$. Since $\kappa$ is continuous up to the initial time $\alpha_\lambda$, we have $T^\lambda>\alpha_\lambda$. We claim that \eqref{eq:1k} holds for $t<T^\lambda$. Indeed, it is satisfied on the initial timeslice $\Gamma^\lambda_{\alpha_\lambda}=\mathrm{A}^\lambda_{t_\lambda}$ since
\[
\kappa_s^2-\kappa^2=\lambda^2\left(\cos^2\theta\sin^2\theta-\sin^2\theta-a^2_\lambda\right)=-\lambda^2(\sin^4\theta+a_\lambda^2)\le 0\,,
\]
whereas $\inner{\gamma}{\nu}\le 0$. We will show that
\[
f_\varepsilon\doteqdot\vert\kappa_s\vert-\kappa+\inner{\gamma}{\nu}-\varepsilon\mathrm{e}^{t-\alpha_\lambda}
\]
remains negative up to time $T^\lambda$. Suppose, to the contrary, that $f_\varepsilon$ reaches zero at some time $t<T^\lambda$ at some point $p\in\overline\Gamma_t$. Since $\vert\kappa_s\vert-\kappa+\inner{\gamma}{\nu}$ vanishes at the boundary, $p$ must be an interior point. Since $\kappa_s$ vanishes at the $y$-axis, and the curve is symmetric, we may assume that $x(p)>0$. At such a point,
\bann
0\le (\pd_t-\Delta)f_\varepsilon={}&\kappa^2(4\kappa_s-\kappa+\inner{\gamma}{\nu})-2\kappa-\varepsilon\mathrm{e}^{t-\alpha_\lambda}\\
={}&\kappa^2(3[\kappa-\inner{\gamma}{\nu}]+4\varepsilon\mathrm{e}^{t-\alpha_\lambda})-2\kappa-\varepsilon\mathrm{e}^{t-\alpha_\lambda}\,.
\eann
Recalling \eqref{eq:gamma dot nu le kappa} and estimating $\kappa\le \frac{1}{2}$ yields
\[
0\le 6\kappa^3-2\kappa+(4\kappa^2-1)\varepsilon\mathrm{e}^{t-\alpha_\lambda}<0\,,
\]
which is absurd. So $f_\varepsilon$ does indeed remain negative, and taking $\varepsilon\to 0$ yields \eqref{eq:crude kappa gradient estimate} for $t<T^\lambda$.

Since $\mathrm{Length}(\Gamma_t\cap\{x\ge 0\})\le 1$, integrating \eqref{eq:crude kappa gradient estimate} 
yields
\[
\overline\kappa\le \mathrm{e}^2\underline\kappa\;\;\text{for}\;\; t<T^\lambda\,.
\]
Recalling \eqref{eq:kappa bounds} and \eqref{sin theta upper bound}, this implies that
\[
\overline\kappa\le \mathrm{e}^2\frac{\mathrm{e}^{t}}{\sqrt{1-\mathrm{e}^{2t}}}\;\;\text{for}\;\; t<T^\lambda\,.
\]
Taking $t=T^\lambda$ we find that $T^\lambda\ge T$, where $T$ is independent of $\lambda$, so we conclude that
\[
\overline\kappa\le C\mathrm{e}^{t}\;\;\text{for}\;\; t<T\,,
\]
where $C$ and $T$ do not depend on $\lambda$.
\end{proof}

\begin{lemma}
There exist $C<\infty$ and $T>-\infty$ such that
\[
\frac{\kappa}{y}\le \lambda_0^2+C\mathrm{e}^{2t}\;\;\text{for}\;\; t<T\,.
\]
\end{lemma}
\begin{proof}
Consider the old-but-not-ancient solution $\{\Gamma^\lambda_t\}_{t\in(-\infty,0)}$. By \eqref{eq:curvature estimate}, we can find $C<\infty$ and $T>-\infty$ such that
\bann
(\pd_t-\Delta)\frac{\kappa}{y}={}&\kappa^2\frac{\kappa}{y}+2\inner{\cd\frac{\kappa}{y}}{\frac{\cd y}{y}}\\
\le{}&C\mathrm{e}^{2t}\frac{\kappa}{y}+2\inner{\cd\frac{\kappa}{y}}{\frac{\cd y}{y}}\;\;\text{for}\;\; t<T\,.
\eann
Since, at a boundary point,
\[
\left(\frac{\kappa}{y}\right)_s=\frac{\kappa_s}{y}-\frac{\kappa}{y}\frac{y_s}{y}=0\,,
\]
the Hopf boundary point lemma 
and the \textsc{ode} comparison principle yield
\[
\max_{\Gamma^\lambda_t}\frac{\kappa}{y}\le C\max_{\Gamma^\lambda_{\alpha_\lambda}}\frac{\kappa}{y}\;\;\text{for}\;\; t\in(\alpha_\lambda,T)\,.
\]

But now
\bann
(\pd_t-\Delta)\frac{\kappa}{y}\le{}&C\mathrm{e}^{2t}\max_{\Gamma^\lambda_{\alpha_\lambda}}\frac{\kappa}{y}+2\inner{\cd\frac{\kappa}{y}}{\frac{\cd y}{y}}\;\;\text{for}\;\; t<T\,,
\eann
and hence, by \textsc{ode} comparison,
\bann
\max_{\Gamma^\lambda_t}\frac{\kappa}{y}\le \max_{\Gamma^\lambda_{\alpha_\lambda}}\frac{\kappa}{y}\left(1+C\mathrm{e}^{2t}\right)\;\;\text{for}\;\; t\in(\alpha_\lambda,T)\,.
\eann
Since, on the initial timeslice $\Gamma^\lambda_{\alpha_\lambda}=\mathrm{A}_{t_\lambda}^\lambda$,
\[
\frac{\kappa}{y}=\frac{\lambda\tan(\lambda y)}{y}\cos\theta\,,
\]
the claim follows upon taking $\lambda\to\lambda_0$.
\end{proof}

It follows that
\[
\left(\log\underline y(t)-\lambda_0^2t\right)_t\le C\mathrm{e}^{2t}\;\;\text{for}\;\; t<T
\]
and hence, integrating from time $t$ up to time $T$,
\[
\log\underline y(t)-\lambda_0^2t\ge\log\underline y(T)-\lambda_0^2T-C\;\;\text{for}\;\; t<T\,.
\]
So we indeed find that
\begin{lemma} \label{lem:backwards asymptotics}
the limit
\begin{equation}\label{eq:backwards asymptotics}
A\doteqdot\lim_{t\to-\infty}\mathrm{e}^{-\lambda_0^2t}\underline y(t)
\end{equation}
exists in $(0,\infty)$ on the particular ancient solution $\{\Gamma_t\}_{(-\infty,0)}$.
\end{lemma}

\section{Uniqueness}

Now let $\{\Gamma_t\}_{t\in(-\infty,0)}$, $\Gamma_t=\pd_{\mathrm{rel}}\Omega_t$, be \emph{any} convex, locally uniformly convex ancient free boundary curve shortening flow in the disc. 
By Stahl's theorem \cite{Stahl2}, we may assume that $\Gamma_t$ contracts to a point on the boundary as $t\to 0$.

\subsection{Backwards convergence}

We first show that $\overline\Gamma_t$ converges to a bisector as $t\to-\infty$.
\begin{lemma}\label{lem:backwards convergence}
Up to a rotation of the plane,
\[
\overline\Gamma_t\underset{C^\infty}{\longrightarrow}[-1,1]\times\{0\}\,\,\text{as}\;\;t\to-\infty\,.
\]
\end{lemma}
\begin{proof}
Set $A(t)\doteqdot \area(\Omega_t)$. Integrating the variational formula for area 
yields
\[
A(t)=\int_t^0\!\!\!\int_{\Gamma_t}d\theta\,,
\]
where $\theta$ is the turning angle. Since convexity ensures that the total turning angle $\int_{\Gamma_t}d\theta$ is increasing and $A(t)\le \pi$ for all $t$, we find that
\[
\int_{\Gamma_t}d\theta\to 0\,\,\text{as}\;\;t\to-\infty\,.
\]
Monotonicity of the flow, the free boundary condition and convexity now imply that the enclosed regions $\Omega_t$ satisfy
\[
\overline\Omega_t\to B^2\cap \{y\ge 0\}\,\,\text{as}\;\;t\to-\infty
\]
in the Hausdorff topology.

If we now represent $\Gamma_t$ graphically over the $x$-axis, then convexity and the boundary condition ensure that the height and gradient are bounded by the height at the boundary. Stahl's estimates \cite{Stahl1} now give bounds for $\kappa$ and its derivatives up to the boundary depending only on the height at the boundary. We then get smooth subsequential convergence along any sequence of times $t_j\to-\infty$. The claim follows since any sublimit is the horizontal segment.
\end{proof}

We henceforth assume, without loss of generality, that the backwards limit is the horizontal bisector.

\subsection{Reflection symmetry}

We can now prove that the solution is reflection symmetric using Alexandrov reflection across lines through the origin (see Chow and Gulliver \cite{ChGu01}).

\begin{lemma}\label{lem:reflection}
$\Gamma_t$ is reflection symmetric about the $y$-axis for all $t$.
\end{lemma}
\begin{proof}
Given any $\o\in (0, \pi)$, we define the halfspace
\[
H_\omega=\{(x, y): (x, y)\cdot (\sin\o, -\cos\o)>0\}
\]
and denote by $R_\omega$ the reflection about $\partial H_\o$. We first claim that, for every $\omega$, there exists $t= t_\omega$ such that
\[
(R_\o\cdot \Gamma_t)\cap (\Gamma_t\cap H_\o)=\emptyset\;\;\text{for all}\;\; t<t_\o\,.
\]
Assume that the claim is not true. Then there exists $\o\in (0, \pi)$, a sequence of times $t_i\to -\infty$, and a sequence of pairs of points $p_i, q_i\in \Gamma_{t_i}$ such that $R_\o(p_i)= q_i$. This implies that the line passing through $p_i$ and $q_i$ is parallel to the vector $(\sin\o, -\cos\o)$, so the mean value theorem yields for each $i$ a point $r_i$ on $\Gamma_{t_i}$ where the normal is parallel to $(\cos\o, \sin \o)$. This contradicts Lemma \ref{lem:backwards convergence}.
\end{proof}

\subsection{Asymptotics for the height}



We begin with a lemma.

\begin{lemma}
For all $t<0$,
\[
\k_s>0\;\;\text{in}\;\; \{x>0\}\cap\Gamma_t
\]
and hence
\begin{equation}\label{eq:min y lower bound general}
\frac{\sin\overline\theta}{1+\cos\overline\theta}\le \underline y\,.
\end{equation}
\end{lemma}
\begin{proof}
Choose $T>-\infty$ so that $\k<\frac{2}{7}$ for $t<T$ and, given $\e>0$, set
\[
v_\varepsilon\doteqdot \k_s+\e(1-\langle \gamma, \nu\rangle)\,.
\]
We claim that $v_\varepsilon\ge 0$ in $\{x\ge 0\}\cap (-\infty,T)$. Suppose that this is not the case. Since at the spatial boundary $v_\varepsilon>\varepsilon$, 
and $v_\varepsilon\to \varepsilon$ as $t\to-\infty$, there must exist a first time in $(-\infty,T)$ and an interior point at which $v_\varepsilon=0$. But, at such a point,
\[
\begin{split}
0\ge\left(\partial_t-\Delta\right)v_\varepsilon&= \k^2(\k_s-\e\langle x, \nu\rangle)+3\k^2\k_s+2\e\k\\
&=-\e\k^2-3\e\k^2(1-\langle x, \nu\rangle)+2\e\k\\
&\ge\e(2-7\k)\k\\
&>0\,,
\end{split}
\]
which is absurd. Now take $\e\to 0$ to obtain $\kappa_s\ge 0$ in $\{x\ge 0\}\cap\Gamma_t$ for $t\in (-\infty,T]$. Since $\kappa_s=0$ at the $y$-axis and $\kappa_s=\kappa>0$ at the right boundary point, the strong maximum principle and the Hopf boundary point lemma imply that $\kappa_s>0$ in $\{x>0\}\cap\Gamma_t$ for $t\in (-\infty,T]$. But then Sturm's theorem implies that $\kappa_s$ does not develop additional zeroes up to time $0$.

Having established the first claim, the second follows as in Lemma \ref{lem:crude estimates}.
\end{proof}

\begin{proposition}\label{prop:linear analysis}
If we define $A\in(0,\infty)$ as in \eqref{eq:backwards asymptotics}, then
\[
\mathrm{e}^{\lambda_0^2t}y(x,t)\to A\cosh(\lambda_0x)\;\;\text{uniformly as}\;\; t\to-\infty\,.
\]
\end{proposition}
\begin{proof}
Given $\tau<0$, consider the rescaled height function
\[
y^{\tau}(x,t)\doteqdot \mathrm{e}^{-\lambda_0^2\tau}y(x,t+\tau)\,,
\]
which is defined on the time-translated flow $\{\Gamma^\tau_t\}_{t\in(-\infty,-\tau)}$, where $\Gamma^\tau_t\doteqdot \Gamma_{t+\tau}$. Note that
\begin{equation}\label{eq:heat equation on the flow}
\left\{\begin{aligned}
(\pd_t-\Delta^\tau)y^{\tau}={}&0\;\;\text{in}\;\;\{\Gamma_t^\tau\}_{t\in(-\infty,-\tau)}\\
\inner{\cd^\tau y^\tau}{N}={}&y\;\;\text{on}\;\;\{\pd\Gamma_t^\tau\}_{t\in(-\infty,-\tau)}\,,
\end{aligned}\right.
\end{equation}
where $\cd^\tau$ and $\Delta^\tau$ are the gradient and Laplacian on $\{\Gamma^\tau_t\}_{t\in(-\infty,-\tau)}$, respectively, and $N$ is the outward unit normal to $\pd B^2$. 

Since $\{\Gamma_t\}_{t\in(-\infty,0)}$ reaches the origin at time zero, it must intersect the constructed solution for all $t<0$. In particular, the value of $\underline y$ on the former can at no time exceed the value of $\overline y$ on the latter. But then \eqref{eq:backwards asymptotics} and \eqref{eq:min y lower bound general} yield
\begin{equation}\label{eq:sup bound sublimit}
\limsup_{t\to-\infty}\mathrm{e}^{-\lambda_0^2t}\overline y<\infty\,.
\end{equation}
This implies a uniform bound for $y^\tau$ on $\{\Gamma^\tau_t\}_{t\in(-\infty,T]}$ for any $T\in\R$. 
So Alaoglu's theorem yields a sequence of times $\tau_j\to-\infty$ such that $y^{\tau_j}$ converges in the weak$^\ast$ topology as $j\to\infty$ to some $y^\infty\in L^2_{\mathrm{loc}}([-1,1]\times (-\infty,\infty))$. 
Since convexity and the boundary condition imply a uniform bound for $\cd^\tau y^\tau$ on any time interval of the form $(-\infty,T]$, we may also arrange that the convergence is uniform in space at time zero, say.

Weak$^\ast$ convergence ensures that $y^\infty$ satisfies the problem
\begin{equation}\label{eq:heat equation on the interval}
\left\{\begin{aligned}
y_t={}&y_{xx}\;\;\text{in}\;\; [-1,1]\times(-\infty,\infty)\\
y_x(\pm 1)={}&\pm y(\pm 1)\,.
\end{aligned}\right.
\end{equation}
Indeed, a smooth function 
$y^\tau$ satisfies the boundary value problem \eqref{eq:heat equation on the flow} (and analogously for \eqref{eq:heat equation on the interval}) if and only if
\[
\int_{-\infty}^{-\tau}\!\int_{\Gamma_t^\tau}y^\tau(\pd_t-\Delta^\tau)^\ast\eta=0
\]
for all smooth $\eta$ which are compactly supported in time and satisfy
\[
\cd^\tau\eta\cdot N=\eta\;\;\text{on}\;\;\pd\Gamma^\tau_t \,,
\]
where $(\pd_t-\Delta^\tau)^\ast\doteqdot -(\pd_t+\Delta^\tau)$ is the formal $L^2$-adjoint of the heat operator. Since $\{\Gamma_t^\tau\}_{t\in(-\infty,-\tau)}$ converges uniformly in the smooth topology to the stationary interval $\{[-1,1]\times\{0\}\}_{t\in(-\infty,\infty)}$ as $\tau\to-\infty$, we conclude that 
the limit $y^\infty$ must satisfy \eqref{eq:heat equation on the interval} in the $L^2$ sense (and hence in the classical sense due to the $L^2$ theory for the heat equation). Indeed, by the definition of smooth convergence, we may (after possibly applying a diffeomorphism) parametrize each flow $\{\overline\Gamma{}_t^{\tau_j}\}_{t\in(-\infty,-\tau_j)}$ over $I\doteqdot [-1,1]$ by a family of embeddings $\gamma^j_t:I\times(-\infty,-\tau_j)\to \overline B{}^2$ which converge in $C^\infty_{\mathrm{loc}}(I\times(-\infty,\infty))$ as $j\to\infty$ to the stationary embedding $(x,t)\mapsto xe_1$. Given $\eta\in C^\infty_0(I\times(-\infty,\infty))$ satisfying $\eta_\zeta(\pm1)=\pm\eta$, set $\eta^j\doteqdot \varphi^j\eta$, where $\varphi^j:[-1,1]\times(-\infty,-\tau^j)\to\R$ is defined by
\[
\varphi^j_\zeta+(1-\vert\gamma^j_\zeta\vert)\varphi^j=0\,,\;\; \varphi^j(0,t)=1\,. 
\]
That is, $\varphi^j(\zeta,t)=\mathrm{e}^{s^j(\zeta,t)-\zeta}$, where $s^j(\zeta,t)\doteqdot\int_0^\zeta\vert\gamma^j_\zeta(\xi,t)\vert\,d\xi$. This ensures that $\cd^{\tau^j}\eta^j\cdot N=\eta^j$ at the boundary, and hence
\bann
0={}&\int_{-\infty}^\infty\int_Iy^{\tau_j}(\pd_t-\Delta^{\tau_j})^\ast\eta^jds^jdt\,.
\eann
Since 
$\varphi^j\to 1$ in $C^\infty_{\mathrm{loc}}(I\times(-\infty,\infty))$, a short computation reveals that
\[
0=\int_{-\infty}^\infty\int_Iy^{\infty}(\pd_t-\Delta)^\ast\eta\,d\zeta dt\,.
\]

Finally, we characterize the limit (uniqueness of which implies full convergence, completing the proof). Separation of variables leads us to consider the problem
\[
\left\{\begin{aligned}
-\phi_{xx}={}&\mu\phi\;\;\text{in}\;\; [-1,1]\\
\phi_x(\pm 1)={}&\pm \phi(\pm 1)\,.
\end{aligned}\right.
\]
There is only one negative eigenspace, and its frequency 
turns out to be $k_{-1}=\lambda_0$, with the corresponding mode given by
\[
\phi_{-1}(x)\doteqdot \cosh(k_{-1}x)\,.
\]
%
Thus, recalling \eqref{eq:sup bound sublimit}, we are able to conclude that
\[
y^\infty(x,t)=A\mathrm{e}^{\lambda_0^2t}\cosh(\lambda_0x) 
\]
for some $A\ge 0$. 
In particular,
\[
\mathrm{e}^{-\lambda_0^2\tau_j}y(x,\tau_j)=y^{\tau_j}(x,0)\to A\cosh(\lambda_0x)\;\;\text{uniformly as}\;\; j\to\infty\,.
\]
Now, if $A$ is not equal to the corresponding value on the constructed solution (note that the full limit exists for the latter), then one of the two solutions must lie above the other at time $\tau_j$ for $j$ sufficiently large. But this violates the avoidance principle. 
\end{proof}

\subsection{Uniqueness}

Uniqueness of the constructed ancient solution now follows directly from the avoidance principle.

\begin{proposition}\label{prop:uniqueness}
Modulo time translation and rotation about the origin, there is only one convex, locally uniformly convex ancient solution to free boundary curve shortening flow in the disc.
\end{proposition}
\begin{proof}
Denote by $\{\Gamma_t\}_{t\in(-\infty,0)}$ the constructed ancient solution and let $\{\Gamma'_t\}_{t\in(-\infty,0)}$ be a second ancient solution which, without loss of generality, contracts to the point $(0,1)$ at time $0$. Given any $\t>0$, consider the time-translated solution $\{\Gamma_t^\tau\}_{t\in(-\infty,-\tau)}$ defined by $\Gamma_t^\t=\Gamma'_{t+\tau}$. By Proposition \ref{prop:linear analysis},
\[
\mathrm{e}^{-\lambda_0^2t}y^\tau(x,t)
\to A\mathrm{e}^{\lambda_0^2\tau}\cosh(\lambda_0x)\;\;\text{as}\;\; t\to-\infty
\]
uniformly in $x$. So $\Gamma^\tau_t$ lies above $\Gamma_t$ for $-t$ sufficiently large. The avoidance principle then ensures that $\Gamma^\tau_t$ lies above $\Gamma_t$ for all $t\in(-\infty,0)$. Taking $\tau\to 0$, we find that $\Gamma'_t$ lies above $\Gamma_t$ for all $t<0$. Since the two curves reach the point $(0,1)$ at time zero, they intersect for all $t<0$ by the avoidance principle. The strong maximum principle then implies that the two solutions coincide for all $t$.
\end{proof}

\section{Supplement: nonexistence of rotators}\label{sec:no rotators}

Free boundary curve shortening flow in $B^2$ is invariant under rotations about the origin, so it is natural to seek solutions which move by rotation; that is, solutions $\gamma:(-\frac{L}{2},\frac{L}{2})\times(-\infty,\infty)\to \overline B{}^2$ satisfying
\[
\gamma(\cdot ,t)=\mathrm{e}^{iBt}\gamma(\cdot,0)
\]
for some $B>0$. Differentiating yields the \emph{rotator equation}
\begin{equation}\label{eq:rotator}
\kappa=-B\inner{\gamma}{\tau}.
\end{equation}

It turns out, however, that there are no solutions to \eqref{eq:rotator} 
in $B^2$ satisfying the free boundary condition.

\begin{proof}[Proof of Theorem \ref{thm:no rotators}]
Following Halldorsson \cite{Halldorsson}, we rewrite the rotator equation as the pair of ordinary differential equations
\begin{equation}\label{eq:rotator system}
x'=B+xy\;\;\text{and}\;\;y'=-x^2\,,
\end{equation}
where
\[
x\doteqdot B\inner{\gamma}{\tau}\;\;\text{and}\;\;y\doteqdot B\inner{\gamma}{\nu}\,.
\]
Arc-length parametrized solutions $\gamma$ to the rotator equation \eqref{eq:rotator} can be recovered from solutions to the system \eqref{eq:rotator system} via
\[
\gamma\doteqdot B^{-1}(x+iy)\mathrm{e}^{i\theta}\,,\;\;\theta(s)\doteqdot-\int_0^sx(\sigma)d\sigma\,,
\]
and this parametrization is unique up to an ambient rotation and a unit linear reparametrization, i.e. $(\theta,s)\mapsto (\pm\theta+\theta_0,\pm s+s_0)$\,.

Note that
\[
\vert\gamma\vert=B^{-1}\sqrt{x^2+y^2}\,.
\]
So we seek solutions $(x,y):(-\frac{L}{2},\frac{L}{2})\to B{}^2$ to \eqref{eq:rotator system} satisfying the free boundary condition $(x(\pm \frac{L}{2}),y(\pm \frac{L}{2}))=(\pm B,0)$. 

Let $\gamma$ be such a solution. Since \eqref{eq:rotator system} can be uniquely solved with initial condition $(x(s_0),y(s_0))=(B,0)$ (which corresponds to $\gamma(s_0)\in \pd B^2$ with $\inner{\gamma}{\tau}|_{s_0}=1$), we find that $\gamma$ must be invariant under rotation by $\pi$ about the origin. In particular, the points $\gamma(-\frac{L}{2})$ and $\gamma(\frac{L}{2})$ are diametrically opposite. It follows that $\gamma(0)$ is the origin. Indeed, for topological reasons, $\gamma$ must cross the line orthogonally bisecting the segment joining its endpoints an odd number of times (with multiplicity). But since the rotational invariance pairs each crossing above the origin with one below, we are forced to include the origin in the set of crossings. 
We conclude that
\[
0=y(\tfrac{L}{2})=\int_0^{\frac{L}{2}}y'=-\int_0^\frac{L}{2}x^2ds\,,
\]
which is impossible since $x(\frac{L}{2})=B>0$. This completes the proof.
\end{proof}

\bibliographystyle{acm}
\bibliography{../../bibliography}

\end{document}